\documentclass[a4paper]{amsart}
\usepackage{amssymb}
\usepackage{amscd}
\usepackage{amsthm}
\usepackage{amsmath}
\usepackage{latexsym}
\usepackage[all]{xy}

\theoremstyle{plain}
\newtheorem{theorem}{Theorem}
\newtheorem{corollary}[theorem]{Corollary}
\newtheorem{lemma}[theorem]{Lemma}

\theoremstyle{definition}
\newtheorem{definition}[theorem]{Definition}
\newtheorem{example}[theorem]{Example}

\theoremstyle{remark}

\newtheorem{remark}[theorem]{Remark}
\newtheorem{notation}[theorem]{Notation}
\newtheorem{acknowledgment}[theorem]{Acknowledgment}
\numberwithin{theorem}{section}

\newcommand{\card}[1]{\mbox{\rm{card}}(#1)}

\newcommand{\Soc}[1]{\mbox{\rm{Soc}}(#1)}

\newcommand{\im}[1]{\mbox{\rm{Im}}(#1)}

\newcommand{\Hom}[3]{\mbox{\rm{Hom}}_{#1}(#2,#3)}
\newcommand{\Ext}[4]{\mbox{\rm{Ext}}^{#1}_{#2}(#3,#4)}

\newcommand{\rmod}[1]{\mbox{\rm{Mod}--}{#1}}

\newcommand{\simp}[1]{\mbox{\rm{simp}--}{#1}}

\begin{document}

\title{The Dual Baer Criterion for non-perfect rings}
\author{\textsc{Jan Trlifaj}}
\dedicatory{To Roger and Sylvia Wiegand, in honor of their 150th birthday}
\address{Charles University, Faculty of Mathematics and Physics, Department of Algebra \\
Sokolovsk\'{a} 83, 186 75 Prague 8, Czech Republic}
\email{trlifaj@karlin.mff.cuni.cz}

\date{\today}
\subjclass[2010]{Primary: 16D40, 03E35. Secondary: 16E50, 16D60, 03E45, 18G05.}
\keywords{R-projective module, von Neumann regular ring, semiartinian module, Jensen's Diamond Principle.}
\thanks{Research supported by GA\v CR 17-23112S}
\begin{abstract} Baer's Criterion for Injectivity is a useful tool of the theory of modules. Its dual version (DBC) is known to hold for all right perfect rings, but its validity for the non-right perfect ones is a complex problem (first formulated by Faith in 1976 \cite{F}). Recently, it has turned out that there are two classes of non-right perfect rings: 1.\ those for which DBC fails in ZFC, and 2.\ those for which DBC is independent of ZFC. First examples of rings in the latter class were constructed in \cite{T4}; here, we show that this class contains all small semiartinian von Neumann regular rings with primitive factors artinian.    
\end{abstract}

\maketitle

\section{Introduction}

The celebrated Baer's Criterion for Injectivity \cite{B} enables a restriction to cyclic modules when testing for injectivity of a module over an arbitrary ring. It forms the basic step in the classification of injective modules over various right noetherian rings, the study of injective complexes of modules, etc. 

It is easy to see that the dual version of the Baer Criterion (called DBC for short), which enables restriction to cyclic modules when testing for projectivity of a module, works for all finitely generated modules over any ring \cite[16.14(2)]{AF}. Moreover, DBC holds for all modules over any right perfect ring by \cite{S} (or \cite{KV}). In view of these positive results, Faith \cite[p.175]{F} raised the question of characterizing the (non-right perfect) rings for which DBC holds (for all modules). 

Gradually, various classes of non-right perfect rings were shown to fail DBC, and it has even turned out to be consistent with ZFC that no non-right perfect ring satisfies DBC \cite{AIPY}. On the positive side, examples of particular non-right perfect rings for which it is consistent with ZFC that DBC holds have recently been constructed in \cite{T4}; for those particular rings, the validity of DBC is independent of ZFC. 
 
The recent results mentioned above motivate the following refinement of Faith's question: Find a boundary line between those non-right perfect rings for which DBC fails in ZFC, and those for which it is independent of ZFC. The former class of rings is a large one: it contains all commutative noetherian rings \cite[Theorem 1]{H1}, all semilocal right noetherian rings \cite[Proposition 2.11]{AIPY}, and all commutative domains \cite[Lemma 1]{T4}. Our goal here is to further develop the approach of \cite{T4} to study the latter class. 

Our main Theorem \ref{consistency} shows that the latter class contains a particular kind of transfinite extensions of simple artinian rings, the small semiartinian von Neumann regular rings with primitive factors artinian (cf.\ Definition \ref{smallR}). We also show (in ZFC) that for each cardinal $\kappa$, there exists a non-right perfect ring $R$ such that DBC holds for all $\leq \kappa$-generated modules (see Example \ref{kappa}).   

\medskip 
Let $R$ be a ring. A (right $R$-) module $M$ is called \emph{$R$-projective} provided that each homomorphism from $M$ into $R/I$ where $I$ is any right ideal of $R$, factorises through the canonical projection $\pi : R \to R/I$ \cite[p.184]{AF}. The \emph{Dual Baer Criterion (DBC)} says that a module $M$ is projective, if and only if $M$ is $R$-projective.   

A ring $R$ is \emph{von Neuman regular} provided that for each $r \in R$ there exists $s \in R$ such that $rsr = r$.

A ring $R$ is right \emph{semiartinian} if $\Soc M \neq 0$ for each non-zero module $M$. In this case, there are a non-limit ordinal $\tau$ and a strictly increasing chain of ideals of $R$, $\mathcal S = ( S_\alpha \mid \alpha \leq \tau )$, such that $S_0 = 0$, $S_{\alpha +1}/S_\alpha = \Soc{R/S_\alpha}$ for each $\alpha < \tau$, $S_\alpha = \bigcup_{\beta < \alpha} S_\beta$ for each limit ordinal $\alpha < \tau$, and $S_\tau = R$. The chain $\mathcal S$ is called the (right) \emph{socle sequence} of $R$, and $\tau$ is the (right) \emph{Loewy length} of $R$.  

It is well-known that if $R$ is right semiartinian of Loewy length $\tau$, then each right module $M$ is semiartinian, and has a socle sequence $\mathcal M = ( M_\alpha \mid \alpha \leq \eta )$ for some $\eta \leq \tau$. The ordinal $\eta$ is the \emph{Loewy length} of $M$.
The completely reducible module  $M_{\alpha +1}/M_\alpha = \Soc{M/M_\alpha}$ is called the \emph{$\alpha$th layer} of $M$ ($\alpha < \eta$).
Note that if $M$ is finitely generated (e.g., when $M = R$), then $\eta$ is a non-limit ordinal, and the last (= $(\eta -1)$th) layer of $M$ is finitely generated.   

Also, if $R$ is von Neumann regular, then $R$ is left semiartinian, iff $R$ is right semiartinian, and in this case, the left and right socle sequences of $R$ coincide. 

Let $R$ be a von Neumann regular ring. Then $R$ is said to have \emph{primitive factors artinian} in case $R/P$ is (right) artinian for each (right) primitive ideal $P$ of $R$. Again, this notion is left-right symmetric, and it is equivalent to all prime factors of $R$ being artinian, see \cite[Theorem 6.2]{G}. By \cite[Proposition 6.18]{G}, it is also equivalent to the injectivity of all all homogenous completely reducible modules, i.e., to $\sum$-injectivity of all simple modules. In particular, in this case, all layers of any semiartinian module $M$, with the possible exception of the last one, are infinitely generated; in fact, they even have infinitely many non-zero homogenous components.

For further properties of the notions defined above, we refer the reader to \cite{AF}, \cite{EM}, \cite{GT} and \cite{G}.  

\section{$R$-projective modules}\label{versus}

In order to study relations between $R$-projective and projective modules, it is convenient to recall the more general notions of an $N$-projective module, and a projectivity domain of a module, from \cite[\S 16]{AF}: 

\begin{definition}\label{Mproj} Let $R$ be a ring and $M \in \rmod R$. 
\begin{enumerate}
\item Let $N \in \rmod R$. Then $M$ is \emph{$N$-projective} provided that the functor $\Hom RM{-}$ is exact at each short exact sequence whose middle term is $N$.  
\item The set of all $N \in \rmod R$ such that $M$ is $N$-projective is called the \emph{projectivity domain} of $M$, and denoted by $\mathcal{P}r^{-1}(M)$. 
\end{enumerate}
\end{definition}

Clearly, $M$ is projective, iff $\mathcal{P}r^{-1}(M) = \rmod R$. The following lemma is well-known (see e.g.\ \cite[16.12 and 16.14(2)]{AF}):

\begin{lemma}\label{closureprops} Let $M \in \rmod R$. Then the class $\mathcal{P}r^{-1}(M)$ is closed under submodules, homomorphic images, and finite direct sums. 

In particular, each $R$-projective module is $F$-projective for any finitely generated module $F$. Hence each finitely generated $R$-projective module is projective. 
\end{lemma}

For a right ideal $I$ of a ring $R$, we will denote by $\pi_I : R \to R/I$ the canonical projection. We will need the following observation:

\begin{lemma}\label{summands} Let $R$ be a ring and $I \subseteq J$ be right ideals such that the epimorphism $\rho : R/I \to R/J$ splits.  Let $M \in \rmod R$ be such that each homomorphism from $M$ to $R/I$ factorises through $\pi_I$. Then each homomorphism from $M$ to $R/J$ factorises through $\pi_J$.
\end{lemma} 
\begin{proof} Let $\mu \in \Hom R{R/J}{R/I}$ be such that $\rho \mu = \hbox{id}_{R/J}$. Let $f \in \Hom RM{R/J}$. By assumption, there exists $g \in \Hom RMR$ such that $\pi_I g = \mu f$. Then $\pi_J g = \rho \pi_I g = \rho \mu f = f$, whence $f$ factorises through $\pi_J$.
\end{proof}   

By Lemma \ref{closureprops}, the question of existence of $R$-projective modules that are not projective is of interest only for modules that are not finitely generated. Also, since $R$-projectivity and projectivity coincide for all right perfect rings by \cite{S} (see also \cite[Theorem 1]{KV}), that is, DBC holds for these rings, we can w.l.o.g.\ restrict our investigation to non-right perfect rings. 

For many non-perfect rings, DBC fails because there exist countably generated $R$-projective modules that are not projective:
  
\begin{lemma}\label{count} Let $R$ be a non-right perfect right noetherian ring.   
\begin{enumerate}
\item \cite[Lemma 2.10]{AIPY} \, Let $M \neq 0$ be a module possessing no maximal submodules. Then $M$ is $R$-projective, but not projective.
\item Assume $R$ is commutative and there is a regular element $r \in R$ such that $r$ is not invertible (i.e., the depth of $R$ is $\geq 1$). Let $p$ be a maximal ideal such that $r \in p$. Then $E(R/p)$ is a countably generated $R$-projective module which is not projective.
\end{enumerate}
\end{lemma}
\begin{proof} (1) Since $M$ has no maximal submodules, it has no non-zero finitely generated homomorphic images. Since $R$ is right noetherian, all finitely generated submodules of cyclic modules are also finitely generated. Hence $\Hom RM{R/I} = 0$ for each right ideal $I$ of $R$, whence $M$ is $R$-projective. However, $M$ is not projective, as non-zero projective modules contain maximal submodules (see e.g. \cite[17.14]{AF}). 

(2) By part (1), it suffices to prove that $E(R/p)$ has no maximal submodules. Assume this is not the case. Note that $E(R/p)$ is (countably) $R/p$-filtered, i.e., $E(R/p)$ is the union of a continuous (countable) increasing chain of its submodules with consecutive factors isomorphic to $R/p$. The same holds for any non-zero factor-module $F$ of $E(R/p)$. Since $p$ is the only associated prime of $R/p$, the same is true of $F$, whence the only simple factor-module of $E(R/p)$ is $R/p$. As $E(R/p)$ is injective, and hence divisible, it follows that $R/p$ is a divisible module. But $(R/p).r = 0$, a contradiction. 
\end{proof}

\begin{remark}\label{rem0} A closer look at the proofs of Lemmas 1, 2, and the Theorem in \cite{H2}, shows that if $R$ is a commutative noetherian ring such that each countably generated module has a maximal submodule, then the Jacobson radical $J$ of $R$ is $T$-nilpotent and $R/J$ is von Neumann regular, whence $R$ is artinian. 

So a countably generated $R$-projective, but not projective, module exists over \emph{any} commutative noetherian ring which is not artinian. In Lemma \ref{countable}, we will see that this fails for commutative rings in general.   
\end{remark}

Next, we will see that for \emph{each} infinite cardinal $\kappa$, there exists a non-right perfect ring $R$ such that for \emph{each} $\leq \kappa$-generated module $M$, the projectivity of $M$ is equivalent to its $R$-projectivity. In other words, DBC holds for all $\leq \kappa$-generated modules:

\begin{example}\label{kappa} Let $K$ be a skew-field, $\kappa$ an infinite cardinal, and $R$ be the endomorphism ring of a $\kappa$-dimensional left vector space $V$ over $K$. Let $B$ be a left $K$-basis of $V$. Since $\kappa$ is infinite, $B$ is a disjoint union of its subsets $\{ B_\alpha \mid \alpha < \kappa \}$ such that $B_\alpha$ has cardinality $\kappa$ for all $\alpha < \kappa$. For each $\alpha < \kappa$, let $b_\alpha : B_\alpha \to B$ be a bijection, and $r_\alpha \in R$ be such that $r_\alpha \restriction B_\alpha = b_\alpha$ and $r_\alpha \restriction B_\beta = 0$ for all $\alpha \neq \beta < \kappa$. 

Since $b_\alpha$ is a bijection, $r_\alpha$ is left invertible in $R$ for each $\alpha < \kappa$. Since the kernel of $r_\alpha$ coincides with the left $K$-subspace of $V$ generated by $\bigcup_{\alpha \neq \beta < \kappa} B_\beta$, and its image is $V$, the set $\{ r_\alpha \mid \alpha < \kappa \}$ is a right $R$-independent subset of $R$. So the right ideal $I = \sum_{\alpha < \kappa} r_\alpha R$ in $R$ is a free module of rank $\kappa$.  

Assume that $M$ is an $R$-projective $\leq \kappa$-generated module. Then $I \in \mathcal{P}r^{-1}(M)$ by Lemma \ref{closureprops}, so $M$ is $R^{(\kappa )}$-projective, and hence projective. 
\end{example}

However, it is consistent with ZFC that there exists \emph{no} non-right perfect ring $R$ such that DBC holds for \emph{all} modules. Namely, as observed in \cite{AIPY}, \cite[Lemma 2.4]{T2} (or \cite{T1}) prove that the following holds in the extension of ZFC satisfying Shelah's Uniformization Principle (SUP): 

\begin{lemma}\label{sup} Assume SUP. Let $R$ be a non-right perfect ring. Then there exists a module $N$ of projective dimension $1$ such that $\Ext 1RN{I} = 0$ for each right ideal $I$ of $R$. Hence $N$ is $R$-projective, but not projective. In particular, DBC fails in $\rmod R$.
\end{lemma}  

Note that if ZFC is consistent, then so is ZFC + GCH + SUP, see \cite[\S2]{ES}.   

\medskip
Suprisingly, there do exist other extensions of ZFC, and non-right perfect rings $R$ with the property that $R$-projectivity and projectivity 
coincide for \emph{all} modules. In order to provide more details, we need to recall the relevant set-theoretic concepts:

Let $\kappa$ be a regular uncountable cardinal. A subset $C \subseteq \kappa$ is called a \emph{club} provided that $C$ is \emph{closed} in $\kappa$ (i.e., $\hbox{sup}(D) \in C$ for each subset $D \subseteq C$ such that $\hbox{sup}(D) < \kappa$) and $C$ is \emph{unbounded} (i.e., $\hbox{sup}(C) = \kappa$). Equivalently, there exists a strictly increasing continuous function $f : \kappa \to \kappa$ whose image is $C$. A subset $E \subseteq \kappa$ is \emph{stationary} provided that $E \cap C \neq \emptyset$ for each club $C \subseteq \kappa$.

Let $A$ be a set of cardinality $\leq \kappa$. An increasing continuous chain, $\{ A_\gamma \mid \gamma < \kappa \}$, consisting of subsets of $A$ of cardinality $< \kappa$ such that $A = \bigcup_{\gamma < \kappa} A_\gamma$, is called a \emph{$\kappa$-filtration} of the set $A$. Similarly \cite[IV.1.3.]{EM}, if $M$ is a $\leq \kappa$-generated module, then an increasing continuous chain, $( M_\gamma \mid \gamma < \kappa )$, consisting of $< \kappa$-generated submodules of $M$ such that $M = \bigcup_{\gamma < \kappa} M_\gamma$, is called a \emph{$\kappa$-filtration} of the module $M$.  

The key notion needed here is a variant of the notion of Jensen-functions, \cite{J}, \cite[\S 18.2]{GT}, defined as follows: 

Let $\kappa$ be a regular uncountable cardinal and $E$ be a stationary subset of $\kappa$. Let $A$ and $B$ be sets of cardinality $\leq \kappa$. Let $\{ A_\gamma \mid \gamma < \kappa \}$ be a $\kappa$-filtration of $A$. For each $\gamma < \kappa$, let $c_\gamma: A_\gamma \to B$ be a map. Then $( c_\gamma \mid \gamma < \kappa )$ are called \emph{Jensen-functions} provided that for each map $c : A \to B$, the set $E(c) = \{ \gamma \in E \mid c \restriction A_\gamma = c_\gamma \}$ is stationary in $\kappa$.

The existence of Jensen-functions is equivalent to the validity of the well known Jensen's Diamond Principle $\diamondsuit$, that is, to the validity of $\diamondsuit_{\kappa}(E)$ for each regular uncountable cardinal $\kappa$ and each stationary subset $E \subseteq \kappa$ (see e.g.\ \cite[\S 18.2]{GT} for more details):

\begin{lemma}\label{jensen} Assume $\diamondsuit$. Then Jensen-functions exist, that is, the following assertion holds true: 

For each regular uncountable cardinal $\kappa$, $E \subseteq \kappa$ a stationary subset of $\kappa$, and $A$ and $B$ sets of cardinality $\leq \kappa$ with $A$ equipped with a $\kappa$-filtration $\{ A_\gamma \mid \gamma < \kappa \}$, there exist Jensen-functions $( c_\gamma \mid \gamma < \kappa )$, i.e., $c_\gamma : A_\gamma \to B$ for each $\gamma < \kappa$, and for each map $c : A \to B$, the set $E(c) = \{ \gamma \in E \mid c \restriction A_\gamma = c_\gamma \}$ is stationary in $\kappa$. 
\end{lemma}

Jensen \cite{J} proved that $\diamondsuit$ holds under the assumption of G\" odel's Axiom of Constructibility (V = L). Note also that $\diamondsuit$ implies GCH (that is, $2^\lambda = \lambda^+$ for each infinite cardinal $\lambda$), see e.g.\ \cite[18.14]{GT}. Hence if ZFC is consistent, then so is ZFC + GCH + $\diamondsuit$. 

By \cite[Theorem 3.3]{T4}, the following holds in the extension of ZFC satisfying $\diamondsuit$: 

\begin{lemma}\label{jf} Assume $\diamondsuit$. Let $K$ be a field of cardinality $\leq 2^{\aleph_0}$ and $R = ECS(K)$ be the subring of $K^\omega$ consisting of all eventually constant sequences. Then each $R$-projective module is projective, i.e., DBC holds in $\rmod R$. 
\end{lemma}

Hence, by Lemmas \ref{sup} and \ref{jf}, the question of whether DBC holds for the ring $R = ECS(K)$ is independent of ZFC.    

\medskip
Our goal here is to further develop the approach of \cite{T4} and study $R$-projectivity, its generalizations, and DBC, in the much larger class of semiartinian von Neumann regular rings with primitive factors artinian.

\section{Semiartinian von Neumann regular rings with primitive factors artinian}\label{Loewy}
 
We start with recalling the following characterization from \cite[Theorem 2.1 and Proposition 2.6]{RTZ}. It provides a more concrete picture of the rings studied here as kind of transfinite extensions of simple artinian rings:

\begin{lemma}\label{semiart} Let $R$ be a right semiartinian ring and $\mathcal S = (S_\alpha \mid \alpha \leq \sigma + 1)$ be the right socle sequence of $R$. Then the following conditions are equivalent:
\begin{enumerate}
\item $R$ is von Neumann regular and has primitive factors artinian,
\item for each $\alpha \leq \sigma$ there are a cardinal
$\lambda_\alpha$, positive integers $n_{\alpha\beta}$ ($\beta <
\lambda_\alpha$) and skew-fields $K_{\alpha\beta}$ ($\beta <
\lambda_\alpha$) such that $S_{\alpha + 1}/S_\alpha \cong
\bigoplus_{\beta < \lambda_\alpha} M_{n_{\alpha\beta}}(K_{\alpha\beta})$,
as rings without unit. The pre-image of $M_{n_{\alpha\beta}}(K_{\alpha\beta})$ in this isomorphism
coincides with the $\beta$th homogenous component of $\Soc{R/S_\alpha}$,
and it is finitely generated as right $R/S_\alpha$-module
for all $\beta < \lambda_\alpha$.
Moreover, $\lambda_\alpha$ is infinite iff $\alpha < \sigma$.
\end{enumerate}
In this case, $R$ is also left semiartinian, and $\mathcal S$ is the left socle sequence of $R$. 

Moreover, if $P_{\alpha\beta}$ denotes a representative of all simple modules in the $\beta$th homogenous component of $\Soc{R/S_\alpha} = S_{\alpha +1}/S_\alpha$, then $\simp R = \{ P_{\alpha \beta} \mid \alpha \leq \sigma, \beta < \lambda_\alpha \}$ is a representative set of all simple modules. The modules $P_{\alpha\beta}$ ($\beta < \lambda_\alpha$) are called the \emph{$\alpha$th layer simple modules} of $R$. All simple modules are $\sum$-injective, and $\{ P_{0\beta} \mid \beta < \lambda_0 \}$ is a representative set of all projective simple modules.      
\end{lemma}

For example, the ring $R = ECS(K)$ from Lemma \ref{jf} fits the setting of Lemma \ref{semiart} for $\sigma = 1$, $\lambda_0 = \aleph_0$, $\lambda_1 = 1$, and $n_{0\beta} = n_{10} = 1$, $K_{0\beta} = K_{10} = K$ for all $\beta < \lambda_0$.

\medskip
We now fix our notation for the rest of this paper:

\begin{notation}\label{N1}
$R$ will denote a semiartinian von Neumann regular ring with primitive factors artinian, and $\mathcal S = (S_\alpha \mid \alpha \leq \sigma + 1)$ will be its socle sequence. 

For each $\alpha \leq \sigma$, the cardinal $\lambda_\alpha$ from Lemma \ref{semiart} equals the number of homogenous components of the $\alpha$th layer $S_{\alpha + 1}/S_\alpha$ of $R$, so it is an invariant of $R$. Similarly, for each $\beta < \lambda_\alpha$, $n_{\alpha \beta}$ is the dimension of the $\beta$th homogenous component of this layer, and $K_{\alpha \beta}$ is the endomorphism ring of $P_{\alpha, \beta}$. So the $n_{\alpha \beta}$ ($\beta < \lambda_\alpha$) and $K_{\alpha \beta}$ ($\beta < \lambda_\alpha$) are also invariants of $R$, unique up to an ordering of the homogenous components of $S_{\alpha + 1}/S_\alpha$ ($\alpha \leq \sigma$).
\end{notation}       

\begin{corollary}\label{factors} Let $R$ be a semiartinian von Neumann regular ring with primitive factors artinian and $\mathcal S = (S_\alpha \mid \alpha \leq \sigma + 1)$ be its socle sequence. Then for each $\alpha \leq \sigma$, $\bar R = R/S_\alpha$ is a semiartinian von Neumann regular ring with primitive factors artinian, and $\bar{\mathcal S} = (S_\beta/S_\alpha \mid \alpha \leq \beta \leq \sigma + 1)$ is the socle sequence of $\bar R$. Moreover, $\simp {\bar R} = \{ P_{\gamma \beta} \mid \alpha \leq \gamma \leq \sigma, \beta < \lambda_\gamma )$, and $\{ P_{\alpha \beta} \mid \beta < \lambda_{\alpha} \}$ is a representative set of all projective simple $\bar R$-modules.
\end{corollary} 

\medskip
Since we are interested in non-perfect rings, we will further assume that $R$ is not completely reducible, i.e., $\sigma \geq 1$. There is a handy test for $R$-projectivity in the case when the Loewy length of $R$ is finite:

\begin{theorem}\label{finlen} Assume $\sigma \geq 1$ is finite. For each $0 < \alpha \leq \sigma$, let $\pi_\alpha : S_{\alpha+1} \to S_{\alpha +1}/S_\alpha$ be the canonical projection. Then the following conditions are equivalent for a module $M$:
\begin{enumerate}
\item $M$ is $R$-projective,
\item For each $0 < \alpha \leq \sigma$ and each $\varphi \in \Hom RM{S_{\alpha +1}/S_\alpha}$, there exists $\psi \in \Hom RM{S_{\alpha +1}}$ such that $\varphi = \pi_\alpha \psi$. 
\end{enumerate}
\end{theorem}
\begin{proof} (1) implies (2) by Lemma \ref{closureprops}.

Assume (2). For a right ideal $J \neq R$, denote by $n(J)$ the greatest $\alpha \leq \sigma$ such that $S_\alpha \subseteq J$. By downward induction on $\alpha \leq \sigma$, we will prove that for each $J$ with $n(J) = \alpha$, each $f \in \Hom RM{R/J}$ factorises through the canonical projection $\pi_J \in \Hom RR{R/J}$. For $\alpha = \sigma$, this follows from (2) using Lemma \ref{summands} for $I = S_\sigma$ (since $S_{\sigma +1} = R$, and $R/S_\sigma$ is completely reducible). 

For the inductive step, consider a right ideal $J$ with $n(J) = \alpha < \sigma$ and let $f \in \Hom RM{R/J}$. Denote by $\rho : R/J \to R/(S_{\beta +1}+J)$ the canonical projection. Since $n(S_{\alpha +1} + J) \geq \alpha +1$, the inductive premise yields $h \in \Hom RMR$ such that $\pi_{S_{\alpha +1} + J}h = \rho f$. 

In order to prove that $f$ factorises through $\pi_J$, it suffices to factorise $\delta = f - \pi_J h$ through $\pi_J$. Notice that $\rho \delta = 0$, so $\delta$ maps into $({S_{\alpha +1}}+J)/J$, whence it suffices to factorise $\delta$ through $\pi_J \restriction S_{\alpha +1}$. Denoting by $\eta$ the canonical isomorphism $({S_{\alpha +1}+J})/J \to S_{\alpha +1}/(S_{\alpha +1} \cap J)$, we infer that it suffices to factorise $\eta \delta$ through the canonical projection $\rho : S_{\alpha +1} \to S_{\alpha +1}/(S_{\alpha +1} \cap J)$. 

Since the module $S_{\alpha +1}/S_\alpha$ is completely reducible, the canonical projection $\pi : S_{\alpha +1}/S_\alpha \to S_{\alpha +1}/(S_{\alpha +1} \cap J)$ splits, i.e., there is $\mu : S_{\alpha +1}/(S_{\alpha +1} \cap J) \to S_{\alpha +1}/S_\alpha$ such that $\pi \mu = 1$.
By (2), $\mu \eta \delta$ factorizes through $\pi_\alpha$, whence $\eta \delta = \pi \mu \eta \delta$ factorizes through $\rho = \pi \pi_\alpha$, q.e.d. 
\end{proof} 

\begin{remark}\label{rema} Note that the implication (1) implies (2) of Theorem \ref{finlen} holds in general, without any restriction on the Loewy length $\sigma$, by Lemma \ref{closureprops}.
\end{remark}

It will be convenient to consider weaker notions of $R$-projectivity. The first one is obtained by weakening Condition (2) of Theorem \ref{finlen} (which holds in general, for any $R$-projective module by Remark \ref{rema}), the second by a condition on the form of the layers:   

\begin{definition}\label{weakRproj} Let $M$ be a module.
{\begin{enumerate}
\item $M$ is \emph{weakly $R$-projective} provided that for each $0 < \alpha \leq \sigma$ and each $\varphi \in \Hom RM{S_{\alpha +1}/S_\alpha}$ such that $\im \varphi$ is finitely generated, there exists $\psi \in \Hom RM{S_{\alpha +1}}$ such that $\varphi = \pi_\alpha \psi$. 
\item Let $\tau \, (\leq \sigma + 1)$ be the Loewy length of $M$. Then $M$ is \emph{layer projective} provided that for each $0 < \alpha < \tau$, the $\alpha$th layer of $M$ is a projective $R/S_\alpha$-module (i.e., the $\alpha$th layer of $M$ is isomorphic to a direct sum of copies of the simple modules $P_{\alpha \beta}$ ($\beta < \lambda_{\alpha}$), see Corollary \ref{factors}). 
\end{enumerate}}
\end{definition}

Since $R$ is von Neumann regular, each projective module is isomorphic to a direct sum of cyclic modules generated by idempotents of $R$, whence \cite[Theorem 3.8(i)]{T3} implies that each projective module is layer projective. This property can be extended to all weakly $R$-projective modules:

\begin{lemma}\label{rough} Let $M$ be a module.
{\begin{enumerate}
\item For each $\alpha \leq \sigma$, $MS_{\alpha + 1}/MS_\alpha$ is isomorphic to a direct sum of $\alpha$th layer simple modules of $R$.  
\item Assume $M$ is weakly $R$-projective. Then $M$ is layer projective.
\item The class of all weakly $R$-projective modules is closed under submodules. 
\item The class of all layer projective modules is closed under submodules.
\end{enumerate}}
\end{lemma}
\begin{proof} (1) For $\alpha \leq \sigma$ and each $x \in M$, there is a homomorphism $f_x : S_{\alpha +1}/S_\alpha \to MS_{\alpha +1}/MS_\alpha$ defined by $f_x(s + S_\alpha ) = xs + MS_\alpha$. Since $\im{f_x}$ is isomorphic to a direct sum of $\alpha$th layer simple modules of $R$, so is the factor-module $MS_{\alpha +1}/MS_\alpha = \sum_{x \in M} \im{f_x}$.

(2) Let $\mathcal M = ( M_\beta \mid \beta \leq \alpha )$ be the socle sequence of $M$ (so $\alpha \leq \sigma + 1$). By part (1), it suffices to prove that that $M_\beta = MS_\beta$ for each $\beta \leq \alpha$. We will prove this by induction on $\beta$. The claim is clear for $\beta = 0$. For the inductive step, we first use part (1) to get $MS_{\beta +1}/MS_\beta = MS_{\beta +1}/M_\beta \subseteq M_{\beta + 1}/M_\beta$. 

Assume $MS_{\beta +1}/M_\beta \neq M_{\beta + 1}/M_\beta$, i.e., $MS_{\beta +1}/M_\beta \oplus C = M_{\beta + 1}/M_\beta$ for a non-zero completely reducible module $C$. Let $P$ be a simple direct summand of $C$, and $M_\beta \subseteq N \subseteq M_{\beta +1}$ be such that $N/M_\beta = P$. Since $P \in \simp R$, $P$ is isomorphic to a $\gamma$th layer simple module of $R$ for some $\gamma \leq \sigma$ by Lemma \ref{semiart}. Since $M_\beta = MS_\beta$, we have $P.S_\beta = 0$, whence $\gamma \geq \beta$. Since $S_{\beta + 1}$ is generated by idempotents,  $MS_{\beta + 1} \subseteq M_{\beta +1}$ implies $MS_{\beta + 1} = M_{\beta + 1}S_{\beta + 1}$. So if $\gamma = \beta$, then $PS_{\beta + 1} = 0$, a contradiction. Thus $\gamma > \beta$.

As $P$ is an injective direct summand in $S_{\gamma+1}/S_\gamma$, there exists a homomorphism $f : M \to S_{\gamma+1}/S_\gamma$ such that $f \restriction M_\beta = 0$ and $\im f = f(N) = P$. The weak $R$-projectivity of $M$ then gives a $g \in \Hom RM{S_{\gamma+1}}$ such that $\pi_\gamma g = f$. As $\gamma > \beta$, $g(M_{\beta +1}) \subseteq S_{\beta + 1} \subseteq S_\gamma$ and $f(N) = \pi_\gamma g(N) = 0$, a contradiction. This proves that $C = 0$ and $MS_{\beta +1} = M_{\beta + 1}$.

If $\beta$ is a limit ordinal, then $MS_\beta = \bigcup_{\delta < \beta} M.S_\delta = \bigcup_{\delta < \beta} M_\delta = M_\beta$. 

(3) This follows from the fact that for each $0 < \alpha \leq \sigma$, all finitely generated submodules of the $\alpha$th layer module $S_{\alpha +1}/S_\alpha$ are injective.  

(4) Let $N$ be a submodule of a layer projective module $M$ and $\mathcal N = ( N_\beta \mid \beta \leq \tau )$ be the socle sequence of $N$ of Loewy length  $\tau \leq \sigma + 1$. By induction on $\beta$, we will prove that $N_\beta = NS_\beta = N \cap M_\beta$, where $M_\beta = MS_\beta$ is the $\beta$th term of the socle sequence of $M$. 

The only non-trivial part of the proof is the inductive step: we first note, as in part (2), that $NS_{\beta +1}/N_\beta \oplus C = N_{\beta + 1}/N_\beta$ for a completely reducible module $C$. Moreover, $N_{\beta + 1}/N_\beta = N_{\beta + 1}/(N \cap M_\beta) = N_{\beta + 1}/(N_{\beta + 1} \cap M_\beta) \cong  (M_\beta + N_{\beta + 1})/M_\beta \subseteq \Soc {M/M_\beta} = M_{\beta + 1}/M_\beta$. 
By step (1), the latter module is a direct sum of $\beta$th layer simple modules of $R$. So $C = CS_{\beta + 1} \subseteq NS_{\beta +1}/N_\beta$, whence $C = 0$. This proves that $N_{\beta + 1} = NS_{\beta + 1}$. Since $S_{\beta + 1}$ is generated by idempotents, $N \cap M_{\alpha +1} = N \cap MS_{\alpha + 1} = NS_{\alpha + 1}$.    
\end{proof}

The class of all weakly $R$-projective modules is larger than the class of all $R$-projective modules in general, because the latter class need not be closed under submodules:

\begin{example}\label{l3} Let $K$ be a field and $\mathcal A$ be a set of cardinality $2^\omega$ consisting of infinite almost disjoint subsets of $\omega$ (e.g., $\mathcal A$ can be taken as the set of all branches of the tree consisting of all finite sequences of elements of $\omega$). W.l.o.g., we will assume that $\bigcup_{A \in \mathcal A} A = \omega$. 
For each $i < \omega$, let $1_i \in K^\omega$ be the idempotent in $K^\omega$ whose $i$th term is $1$, and all the other terms are $0$. Similarly, for each $A \in \mathcal A$, let $1_A$ be the idempotent in $K^\omega$ whose $i$th term is $1$ for $i \in A$, and all the other terms are $0$. Let $R$ be the unital subalgebra of the $K$-algebra $K^\omega$ generated by $\{ 1_i \mid i < \omega \} \cup \{ 1_A \mid A \in \mathcal A \}$. 

The ring $R$ is commutative semiartinian of Loewy length $3$, with $S_1 = \Soc R$ being the $K$-subspace of $R$ generated by $\{ 1_i \mid i < \omega \}$ and $S_2$ the $K$-subspace of $R$ generated by $\{ 1_i \mid i < \omega \} \cup \{ 1_A \mid A \in \mathcal A \}$. Moreover, $R$ is von Neumann regular with primitive factors artinian, with the invariants  $\lambda_0 = \aleph_0$, $\lambda_1 = 2^\omega$, $\lambda_2 = 1$, $n_{0,\beta} = n_{1 \gamma} = n_{2,0} = 1$ and $K_{0,\beta} = K_{1,\gamma} = K_{2,0} = K$ for all $\beta < \lambda_0$ and $\gamma < \lambda_1$ (see Lemma \ref{semiart}). 

In order to prove that $S_2$ is not $R$-projective, we have to show that condition (2) of Lemma \ref{finlen} fails for $\alpha = 1$. Let $\bar{E} = \Hom R{S_2/S_1}{S_2/S_1}$. Since $S_2/S_1$ is a direct sum of $2^\omega$ pairwise non-isomorphic simple modules, all with the endomorphism ring isomorphic to $K$, there is a $K$-algebra isomorphism $\bar{E} \overset{\phi}\cong K^{2^\omega}$. Further, observe that since $\Hom R{S_1}{S_2/S_1} = 0$, there is a $K$-isomorphism $\Hom R{S_2}{S_2/S_1} \overset{\psi}\cong \bar{E}$ given by $\psi (f) = \bar{f}$, where $f = \bar{f} \pi_1$.  Further, $\Hom R{S_2/S_1}{S_2} = 0$, so the restriction map $g \mapsto g \restriction S_1$ is a monomorphism of the $K$-algebra $E_2 = \Hom R{S_2}{S_2}$ into the $K$-algebra $E_1 = \Hom R{S_1}{S_1}$. Since the elements of $E_1$ are uniquely determined by their values at the idempotents $\{ 1_i \mid i < \omega \}$, $E_1$ is isomorphic to $K^{\omega}$. It follows that for each $g \in E_2$ there is $(k_i)_{i < \omega} \in K^\omega$, such that for each $A \in \mathcal A$ there exists $i_A < \omega$, such that almost all non-zero components of $g(1_A) \in 1_A R \subseteq K^\omega$ equal $k_{i_A}$. So there exists $i < \omega$ such that all but countably many components of $\phi \psi \pi_1 g \in K^{2^\omega}$ equal $k_i$. In particular, $\{ \phi \psi \pi_1 g \mid g \in E_2 \}$ is a proper subset of $K^{2^\omega}$, proving that $S_2$ is not $R$-projective.

Since $S_2 \subseteq R$, we infer that $R$-projective modules are not closed under submodules, and also that $R$ is not hereditary. However, $S_2$ is weakly $R$-projective by Lemma \ref{rough}(3).
\end{example}       
 
Next we concentrate on countably generated modules:

\begin{lemma}\label{countable} Assume $\sigma \geq 1$ is finite. Let $M$ be a weakly $R$-projective module. Then each countably generated submodule of $M$ is projective. 

In particular, a countably generated module is weakly $R$-projective, iff it is projective.
 
\end{lemma}
\begin{proof}  Let $F$ be a finitely generated submodule of $M$. By Lemma \ref{rough}(3), $F$ is weakly $R$-projective. Since $\sigma$ is finite, the weak $R$-projectivity and $R$-projectivity coincide for each finitely generated module by Theorem \ref{finlen}. So $F$ is $R$-projective, and hence projective by Lemma \ref{closureprops}. 

Let $C$ be a countably generated submodule of $M$. By the above, all finitely generated submodules of $C$ are projective, so $C$ is the union of a chain $( F_i \mid i < \omega )$ of finitely generated projective modules. By \cite[Theorem 1.11]{G}, for each $i < \omega$, $F_{i+1} = F_i \oplus P_i$ for a projective module $P_i$, whence $C = F_0 \oplus \bigoplus_{i < \omega} P_i$ is projective. 
\end{proof}
 
The following lemma gives a sufficient condition on $R$ to be hereditary in terms of sizes of the intermediate layers of $R$ (i.e., the $\alpha$th layers of $R$ for $\alpha$ different from $0$ and $\sigma$):  

\begin{lemma}\label{hered} Assume that $R$ has countable Loewy length, and that the module $S_{\alpha+1}/S_\alpha$ is countably generated for each $0 < \alpha < \sigma$. Then $R$ is (left and right) hereditary.
\end{lemma}
\begin{proof} Let $I$ be a right ideal of $R$. Then $I = \bigcup_{\alpha \leq \sigma + 1} (I \cap S_\alpha )$, and the module $(I \cap S_{\alpha +1})/(I \cap S_\alpha)$ is isomorphic to a submodule of the countably generated completely reducible module $S_{\alpha+1}/S_\alpha$ for each $0 < \alpha \leq \sigma$ (for $\alpha = \sigma$, $S_{\alpha+1}/S_\alpha = R/S_\sigma$ is even finitely generated). So the module $I/(I \cap S_1) \cong (I + S_1)/S_1$ is countably generated. 

Since $R$ is von Neumann regular, there exists a countable set of pairwise orthogonal idempotents $\{ e_i \mid i < \omega \}$ in $I$ such that $(\bigoplus_{i < \omega}e_iR) + S_1 = I + S_1$. Moreover,  $((\bigoplus_{i < \omega}e_iR) + S_1)/(\bigoplus_{i < \omega}e_iR ) \cong S_1/ ((\bigoplus_{i < \omega}e_iR) \cap S_1) \cong P$ for a completely reducible and projective module $P \subseteq S_1 = \Soc R$. So $I + S_1 = (\bigoplus_{i < \omega}e_iR) \oplus Q$ where $Q \cong P$, and $I = (\bigoplus_{i < \omega}e_iR) \oplus (I \cap Q)$ is projective, because $I \cap Q$ is a direct summand in $Q$. Similarly, each left ideal of $R$ is projective.
\end{proof}

Next, we record a consequence of the results above for the case of Loewy length $2$, i.e., for $\sigma = 1$ (which includes the rings $R = ECS(K)$ from Lemma \ref{jf}, for example): 

\begin{corollary}\label{two} Assume that $R$ has Loewy length $2$. Then $R$ is hereditary, and a module $M$ is weakly $R$-projective, iff $M$ is $R$-projective, iff each homomorphism from $M$ to $R/\Soc R$ factorises through the canonical projection $\pi : R \to R/\Soc R$. In particular, the class of all $R$-projective modules is closed under submodules. However, there exist layer projective modules that are not weakly $R$-projective.
\end{corollary}
\begin{proof} $R$ is hereditary by Lemma \ref{hered} (since there are no intermediate layers of $R$ for $\sigma = 1$). The characterization of $R$-projective modules follows from Theorem \ref{finlen} (for $n = 1$). Finally, since $R/\Soc R$ is a finitely generated and completely reducible module, it is injective, and the first two claims follow. 

For the last claim, let $F = \{ f_i \mid i < n \}$ be a complete set of orthogonal idempotents in $\bar{R} = R/S_1$, such that $\bar{R} = \bigoplus_{i \leq n} f_i \bar{R}$ and $f_i \bar{R}$ is a simple $\bar{R}$-module for each $i < n$. By \cite[Proposition 2.18]{G}, $F$ lifts in a complete set of orthogonal idempotents $E = \{ e_i \mid i < n \}$ of $R$. In particular, $\Soc{R} = \bigoplus_{i < n} \Soc{e_iR}$ and $R/S_1 = \bigoplus_{i < n} e_iR/\Soc{e_iR}$. For each $i < n$, $e_iR/\Soc{e_iR}$ is a simple non-projective module, whence $\Soc{e_iR} = A_i \oplus B_i$ for some infinitely generated projective completely reducible modules $A_i$ and $B_i$. Let $M_i = e_iR/A_i$. 

We will prove that $M_i$ is layer projective, but not weakly $R$-projective. Clearly, $B_i \subseteq \Soc{M_i}$. Since $B_i$ is not finitely generated, the short exact sequence $0 \to B_i = (A_i \oplus B_i)/A_i \to (e_iR)/A_i \to e_iR/\Soc{e_iR} \to 0$ does not split. Hence $B_i = \Soc{M_i}$, and $M_i$ is layer projective (of length $2$). However, $M_i$ is not weakly $R$-projective, since otherwise $M_i$ is projective by Lemma \ref{countable}, whence $A_i$ is finitely generated, a contradiction.   
\end{proof}  

We finish by summarizing relations between the properties of modules studied in this section: 

\medskip
\noindent projectivity $\implies$ $R$-projectivity $\implies$ weak $R$-projectivity $\implies$ layer projectivity.
\smallskip

By Lemma \ref{sup}, it is consistent with ZFC that the first implication cannot be reversed for any non-right perfect ring $R$. The second implication is actually an equivalence in case $R$ has Loewy length $2$ by Corollary \ref{two}, but it cannot be reversed for rings of Loewy length $3$ in general, by Example \ref{l3}. The third implication cannot even be reversed for any ring of Loewy length $2$ by Corollary \ref{two}. 

In the following section, we will prove that it is consistent with ZFC that the first two implications are equivalences whenever $R$ is small (see Definition \ref{smallR} below).  

\section{Jensen-functions and the projectivity of weakly $R$-projective modules}\label{Faith}

In order to proceed, it will be convenient to characterize weak $R$-projectivity in a more succinct way. First, we need to extend Notation \ref{N1}: 

\begin{notation}\label{N2}
For each $0 < \alpha \leq \sigma$, let $\mathcal F _\alpha$ be the set of all finite subsets of $\lambda_\alpha$. Note that in the setting of Lemma \ref{hered}, $\mathcal F_\alpha$ is countable. For each $F \in \mathcal F _\alpha$, let $N_{\alpha, F}$ be the submodule of $S_{\alpha + 1}$ containing $S_\alpha$, such that $N_{\alpha, F}/S_\alpha \cong \bigoplus_{\beta \in F}M_{n_{\alpha\beta}}(K_{\alpha\beta})$. In other words, $N_{\alpha, F}/S_\alpha$ is the direct sum of those $\beta$th homogenous components of $S_{\alpha + 1}/S_\alpha$ for which $\beta \in F$.  

Note that for each $0 < \alpha \leq \sigma$, each finitely generated submodule of $S_{\alpha + 1}/S_\alpha$ is contained in $N_{\alpha, F}/S_\alpha$ for some $F \in \mathcal F _\alpha$. Let $B = \prod_{0 < \alpha \leq \sigma, F \in \mathcal F _\alpha}N_{\alpha,F}$, $N = \prod_{0 < \alpha \leq \sigma, F \in \mathcal F _\alpha}N_{\alpha , F}/S_\alpha$, and $\pi = \prod_{0 < \alpha \leq \sigma, F \in \mathcal F _\alpha} \pi_\alpha \restriction N_{\alpha,F}$. Notice that $\pi \in \Hom RBN$, and $N$ is an injective module.  
\end{notation}

\begin{lemma}\label{weakproj} A module $M$ is weakly $R$-projective, iff for each $g \in \Hom RMN$ there exists $f \in \Hom RMB$ such that $g = \pi f$.
\end{lemma} 
\begin{proof} This follows directly from Definition \ref{weakRproj}(1), since the homomorphisms $g \in \Hom RMN$ encode in their projections all homomorphisms $h \in \Hom RM{S_{\alpha +1}/S_\alpha}$ ($0 < \alpha \leq \sigma$) with image contained in $N_{\alpha F}/S_\alpha$ ($F \in \mathcal F _\alpha$), i.e., with a finitely generated image.    
\end{proof}
           
Our main result will concern the class of all small rings, defined as follows (cf.\ Notation \ref{N1}):								
								
\begin{definition}\label{smallR} The ring $R$ is \emph{small} provided that $\card{R} \leq 2^\omega$, $R$ has finite Loewy length, and for each $0 < \alpha < \sigma$, the module $S_{\alpha+1}/S_\alpha$ is countably generated.
\end{definition}

Note that if $\card{R} \leq 2^\omega$ and $R$ has Loevy length $2$, then $R$ is small, while Example \ref{l3} shows that rings with $\card{R} \leq 2^\omega$, but of Loewy length $3$, need not be small. By Lemma \ref{hered}, each small ring is hereditary.

\medskip
We are going to show that $\diamondsuit$ implies that weak $R$-projectivity, and hence $R$-projectivity, coincides with projectivity for all small rings. Our proof is a generalization of the proof of \cite[Theorem 3.3]{T4}:   

\begin{theorem}\label{consistency} Assume $\diamondsuit$. Let $R$ be small and $M \in \rmod R$. Then $M$ is weakly $R$-projective, iff $M$ is $R$-projective, iff $M$ is projective. In particular, DBC holds in $\rmod R$. 
\end{theorem}
\begin{proof} We only have to prove that each weakly $R$-projective module $M$ is projective. We do this by induction on the minimal number of generators, $\kappa$, of $M$. For $\kappa \leq \aleph_0$, we can use Lemma \ref{countable}. If $\kappa$ is a singular cardinal, then we employ the fact that the class of all weakly $R$-projective modules is closed under submodules, and apply \cite[Corollary 3.11]{T2}.   

Assume $\kappa$ is a regular uncountable cardinal. Let $A = \{ m_\gamma \mid \gamma < \kappa \}$ be a minimal set of $R$-generators of $M$. For each $\gamma < \kappa$, let $A_\gamma = \{ m_\delta \mid \delta < \gamma \}$. Let $M_\gamma$ be the submodule of $M$ generated by $A_\gamma$. Then $\mathcal M = (M_\gamma \mid \gamma < \kappa )$ is a $\kappa$-filtration of the module $M$. Possibly skipping some terms of $\mathcal M$, we can w.l.o.g.\ assume that $\mathcal M$ has the following property for each $\gamma < \kappa$: if $M_{\delta}/M_\gamma$ is not weakly $R$-projective for some $\gamma < \delta < \kappa$, then also $M_{\gamma + 1}/M_\gamma$ is not weakly $R$-projective. 
Let $E$ be the set of all $\gamma < \kappa$ such that $M_{\gamma + 1}/M_\gamma$ is not weakly $R$-projective.
 
We claim that $E$ is not stationary in $\kappa$. If our claim is true, then there is a club $C$ in $\kappa$ such that $C \cap E = \emptyset$. Let $u : \kappa \to \kappa$ be a strictly increasing continuous function whose image is $C$. For each $\gamma < \kappa$, let $N_\gamma = M_{u(\gamma )}$. Then $(N_\gamma \mid \gamma < \kappa )$ is a $\kappa$-filtration of the module $M$ such that $N_{\gamma + 1}/N_\gamma$ is weakly $R$-projective for all $\gamma < \kappa$. By the inductive premise, $N_{\gamma + 1}/N_{\gamma}$ is projective, hence $N_{\gamma + 1} = N_\gamma \oplus P_\gamma$ for a projective module $P_\gamma$, for each $\alpha < \kappa$. Then $M = N_0 \oplus \bigoplus_{\gamma < \kappa} P_\gamma$ is projective, too.

Assume our claim is not true. Let $B = \prod_{0 < \alpha \leq \sigma, F \in \mathcal F _\alpha}N_{\alpha,F}$ (see Notation \ref{N2}). Since $\card{R} \leq 2^\omega = \aleph_1 \leq \kappa$ by $\diamondsuit$ and by our assumption on $R$, also $\card{B} \leq \aleph_1 \leq \kappa$.  
 
Applying Lemma \ref{jensen} to the setting above, we obtain the Jensen-functions $c_\gamma : A_\gamma \to B$ ($\gamma < \kappa$) such that for each function $c: A \to B$, the set $E(c) = \{ \gamma \in E \mid c_\gamma = c \restriction A_\gamma \}$ is stationary in $\kappa$.

By induction on $\gamma < \kappa$, we will define a sequence $( g_\gamma \mid \gamma < \kappa )$ of homomorphisms, such that $g_\gamma \in \Hom R{M_\gamma}N$ as follows: $g_0 = 0$; if $\gamma < \kappa$, and $g_\gamma$ is defined, we distinguish two cases: 
  
(I) $\gamma \in E$, and there exist $h_{\gamma +1} \in \Hom R{M_{\gamma+1}}N$ and $y_{\gamma +1} \in \Hom R{M_{\gamma+1}}B$, such that $h_{\gamma +1} \restriction M_\gamma = g_\gamma$, $h_{\gamma +1} = \pi y_{\gamma +1}$ and $y_{\gamma +1} \restriction A_{\gamma} = c_\gamma$. In this case we define $g_{\gamma +1} = h_{\gamma +1} + f_{\gamma +1} \rho_{\gamma +1}$, where $\rho_{\gamma +1} : M_{\gamma +1} \to M_{\gamma +1}/M_\gamma$ is the projection and $f_{\gamma +1} \in \Hom R{M_{\gamma+1}/M_\gamma}N$ does not factorise through $\pi$ (such $f_{\gamma +1}$ exists because $\gamma \in E$, so $M_{\gamma+1}/M_\gamma$ is not weakly $R$-projective, by Lemma \ref{weakproj}). Note that $g_{\gamma+1} \restriction M_\gamma = h_{\gamma +1} \restriction M_\gamma = g_\gamma$. 

(II) otherwise. In this case, we extend $g_\gamma$ to $M_{\gamma +1}$ using the injectivity of the module $N$. Thus we obtain $g_{\gamma +1} \in \Hom R{M_{\gamma+1}}N$.
  
If $\gamma < \kappa$ is a limit ordinal, we let $g_\gamma = \bigcup_{\delta < \gamma} g_\delta$, and we define $g = \bigcup_{\gamma < \kappa} g_\gamma$. 

We will prove that $g$ does not factorise through $\pi$. By Lemma \ref{weakproj}, this will contradict the weak $R$-projectivity of $M$, and prove our claim.

Assume there is $f \in \Hom RMB$ such that $g = \pi f$. Let $\gamma \in E(f \restriction A)$. Then $g_{\gamma +1} = \pi (f \restriction M_{\gamma + 1})$ and $f \restriction A_{\gamma} = c_\gamma$. Then $\gamma$ is in case (I) (because $g_{\gamma + 1} \restriction M_\gamma = g_\gamma$, $g_{\gamma +1} = \pi (f \restriction M_{\gamma +1})$, and $(f \restriction M_{\gamma +1}) \restriction A_\gamma = c_\gamma$).  

Let $z_{\gamma +1} = f \restriction M_{\gamma + 1} - y_{\gamma +1}$. Then $z_{\gamma +1} \restriction A_{\gamma} = f \restriction A_{\gamma} - y_{\gamma +1} \restriction A_\gamma = c_\gamma - c_\gamma = 0$. So there exists $x_{\gamma +1} \in \Hom R{M_{\gamma +1}/M_\gamma}B$ such that $z_{\gamma +1} = x_{\gamma +1} \rho_{\gamma +1}$. Moreover, 
$$\pi x_{\gamma +1} \rho_{\gamma +1} = \pi z_{\gamma +1} = \pi f \restriction M_{\gamma + 1} - \pi y_{\gamma +1} = g_{\gamma + 1} - h_{\gamma +1} = f_{\gamma +1} \rho_{\gamma +1}.$$
Since $\rho_{\gamma +1}$ is surjective, we conclude that $\pi x_{\gamma +1} = f_{\gamma +1}$, in contradiction with our choice of the homomorphism $f_{\gamma +1}$.  
\end{proof}           

\begin{corollary}\label{undec} The following assertion is independent of ZFC + GCH: {\lq}If $R$ is a small von Neumann regular semiartinian ring with primitive factors artinian, then all (weakly) $R$-projective modules are projective.{\rq}
\end{corollary}
\begin{proof} Assume SUP. Then the existence of an $R$-projective module $M$ of projective dimension~$1$ follows by Lemma \ref{sup}.

Assume $\diamondsuit$. Then all weakly $R$-projective modules are projective by Theorem \ref{consistency}.
\end{proof}

\begin{remark}\label{best} Though Corollary \ref{undec} covers the case of all rings $R$ with $\card{R} \leq 2^\omega$ and Loewy length $2$, it cannot be extended to all rings $R$ with $\card{R} \leq 2^\omega$ and of Loewy length $3$: Example \ref{l3} shows that such rings may contain weakly $R$-projective modules that are not projective in ZFC.

Also note that by Lemma \ref{rough}(3), if weak $R$-projectivity implies projectivity, then $R$ is necessarily right hereditary. Hence our restriction to countably generated layers in Definition \ref{smallR} (cf.\ Lemma \ref{hered}).   
\end{remark}

\begin{acknowledgment} The author thanks the referee for suggesting several simplifications of the presentation. He also thanks Roger Wiegand for discussions related to the commutative noetherian case (cf.\ Remark \ref{rem0}).
\end{acknowledgment}


\begin{thebibliography}{AIPY}
\bibitem{AIPY}
{H.\ Alhilali, Y.\ Ibrahim, G.\ Puninski, M.\ Yousif},
\textit{When R is a testing module for projectivity?}, 
J. Algebra \textbf{484}(2017), 198-206.
\bibitem{AF}
{F.W.\ Anderson, K.R.\ Fuller}, 
\textit{Rings and Categories of Modules}, 
2nd ed., GTM \textbf{13}, Springer, New York 1992.
\bibitem{B}
{R.\ Baer}, 
\textit{Abelian groups that are direct summands of every
containing abelian group}, Bull. Amer. Math. Soc. \textbf{46}(1940), 800-806.
\bibitem{EM}
{P.C.\ Eklof,  A.H.\ Mekler}, \textit{Almost Free Modules}, 2nd ed., North
Holland Math. Library, Elsevier, Amsterdam 2002.
\bibitem{ES}
{P.C.\ Eklof, S.\ Shelah}, 
\textit{On Whitehead modules}, 
J. Algebra \textbf{142}(1991), 492-510.
\bibitem{F}
{C.\ Faith}, 
\textit{Algebra II.\ Ring Theory}, GMW \textbf{191}, Springer-Verlag, Berlin 1976.
\bibitem{GT}
{R.\ G\"obel, J.\ Trlifaj}, 
\textit{Approximations and Endomorphism Algebras of Modules}, 2nd ed., GEM \textbf{41}, W. de Gruyter, Berlin 2012.
\bibitem{G}
{K.R.\ Goodearl}, 
\textit{Von Neumann Regular Rings}, 2nd ed.,
Krieger Publ.\ Co., Malabar 1991.
\bibitem{H1} 
{R.M.\ Hamsher}, 
\textit{Commutative, noetherian rings over which every module has a maximal submodule}, 
Proc.\ Amer.\ Math.\ Soc. \textbf{17}(1966), 1471-1472.
\bibitem{H2}
{R.M.\ Hamsher}, 
\textit{Commutative rings over which every module has a maximal submodule}, 
Proc.\ Amer.\ Math.\ Soc. \textbf{18}(1967), 1133-1137.
\bibitem{J} 
{R.\ Jensen}, 
\textit{The fine structure of the constructible hierarchy}, 
Ann.\ Math.\ Logic 4(1972), 229-308.
\bibitem{KV} 
{R.D.\ Ketkar, N.\ Vanaja}, \textit{R-projective modules over a semiperfect ring}, Canad.\ Math.\ Bull. \textbf{24}(1981), 365-367.
\bibitem{RTZ} 
{P.\ R{\accent23 u}\v{z}i\v{c}ka, J.\ Trlifaj, J.\ \v{Z}emli\v{c}ka}, \textit{Criteria of steadiness}, Abelian Groups, Module Theory, and Topology, M.Dekker, New York 1998, 359-371.
\bibitem{S} 
{F.\ Sandomierski}, \textit{Relative Injectivity and Projectivity}, PhD thesis, Penn State University, 1964.
\bibitem{T1}
{J.\ Trlifaj}, \textit{Non-perfect rings and a theorem of Eklof and Shelah}, 
Comment. Math. Univ. Carolinae \textbf{32} (1991), 27-32.
\bibitem{T2}
{J.\ Trlifaj}, \textit{Whitehead test modules}, Trans.\ Amer.\ Math.\ Soc.
\textbf{348} (1996), 1521-1554.
\bibitem{T3}
{J.\ Trlifaj},  \textit{Uniform modules, $\Gamma$-invariants and Ziegler spectra of regular rings}, Infinite Abelian Groups and Modules
Trends in Mathematics, Birkh\" auser, Basel 1999, 327-340.
\bibitem{T4}
{J.\ Trlifaj}, \textit{Faith's problem on $R$-projectivity is undecidable},  Proc.\ Amer.\ Math.\ Soc. \textbf{147}(2019), 497-504.
\end{thebibliography}
\end{document}